\documentclass[12pt,reqno]{article}
\usepackage{amsmath} 

\usepackage{color}

\newcommand{\PP}{\mathbb{P}}
\newcommand{\N}{\mathbf N}
\newcommand{\E}{\mathbb{E}}
\newcommand{\B}{\mathbf B}
\newcommand{\T}{\mathbf T}

\usepackage[english]{babel}
\usepackage[latin1]{inputenc}
\usepackage{bbm}
\usepackage{cite}

\usepackage{amssymb}
\usepackage{amscd}

\usepackage{fullpage}
\usepackage{float}

\usepackage{amsthm}
\usepackage{amsfonts}
\usepackage{latexsym}
\usepackage{epsfig}
\usepackage{mathrsfs}
\usepackage{tikz}
\usepackage{pifont}

\theoremstyle{plain}
\newtheorem{theorem}{Theorem}[section]
\newtheorem{lemma}[theorem]{Lemma}

\theoremstyle{definition}
\newtheorem{definition}[theorem]{Definition}

\theoremstyle{remark}
\newtheorem{remark}[theorem]{Remark}

\numberwithin{equation}{section}

\allowdisplaybreaks

\author{
Charles Burnette \\
Department of Mathematics \\
Drexel University \\
Philadelphia, PA 19104-2875 \\
cdb72@drexel.edu
\and
Eric Schmutz \\
Department of Mathematics \\
Drexel University \\
Philadelphia, PA 19104-2875 \\
Eric.Jonathan.Schmutz@drexel.edu
}
\title{Representing Random Permutations as the Product of Two Involutions}

\begin{document}

\thispagestyle{empty}

\maketitle

\begin{abstract}
An involution is a permutation that is its own inverse.  
Given a permutation $\sigma$ of $[n],$ let $\N_{n}(\sigma)$ denote the number 
of ways to write $\sigma$   as a product of two involutions of $[n].$  If we endow  
the symmetric groups $S_{n}$ with uniform probability measures, then
the random variables  ${\mathbf N}_{n}$ are   asymptotically lognormal.  

 The proof is based upon the observation  that, for most permutations $\sigma$, 
$\N_{n}(\sigma)$ can be well approximated  by $\B_{n}(\sigma),$  
the product of the cycle lengths of $\sigma$.   Asymptotic lognormality of 
$\N_{n}$ can therefore be deduced from Erd\H{o}s and Tur\'{a}n's theorem 
that  $\B_{n}$ is itself asymptotically lognormal.
\end{abstract}

\newpage

\section{Introduction}
An involution is a permutation that is its own inverse, i.e. a permutation whose cycle 
lengths are all less than or equal to two.  If $\sigma$ is a permutation of $[n],$
let $\N_{n}(\sigma)$ be the number of ordered pairs of involutions $\tau_{1},\tau_{2}$ of $[n]$
such that $\sigma=\tau_{2}\circ \tau_{1}$. The goal of this paper is to determine the asymptotic 
distribution of the random variable $\N_{n}$ for uniform random permutations $\sigma$.

 Let ${\cal T}_{n}$ be the set of all involutions of $[n]$.  The cardinalities  
$|{\cal T}_{n}|, n=1,2,3,\dots $ have been extensively investigated and form 
OEIS Sequence A000085 \cite{sequence}.  See also Amdeberhan and 
Moll \cite{AM} for more recent work.
 Of  particular importance for this paper is an asymptotic formula that was derived by 
Chowla, Herstein, and Moore \cite{CHM}:
\begin{equation}
\label{asym-tn}
|{\cal T}_{n}| \sim \frac{1}{\sqrt{2}}\!\left(\frac{n}{e}\right)^{n/2}e^{n^{1/2}-1/4}.
\end{equation}
Related approximations appear in Moser and Wyman \cite{MoserWyman},  \cite{MoserWyman-ae}.

Vivaldi and Roberts\cite{RV} studied the
 random permutations that
are obtained by multiplying random involutions with various restrictions on their
fixed points. 
 However the product of two uniformly
random involutions is not a uniformly random permutation. 
For example the identity permutation  is generated with probability $\frac{1}{|{\cal T}_{n}|},$
which is much larger than $\frac{1}{n!}.$ 
   Thus 
$\N_{n}$ is clearly not constant. 

Let
 $I_{\tau_{2},\tau_{1}}(\sigma)=1$ if
 $\tau_{2}\circ \tau_{1}=\sigma$ (and $I_{\tau_{2},\tau_{1}}(\sigma)=0$ otherwise),   so that 
\begin{equation}
\label{sumindicators}
\N_{n}=\sum\limits_{\tau_{1},\tau_{2}}I_{\tau_{2},\tau_{1}}.
\end{equation}
Using this representation and Stirling's formula, it is straightforward to estimate
 the average  number of factorizations \cite{LugoT}:
\begin{equation}
\label{average}
 \E_{n}(\N_{n})=  \frac{1}{n!}\sum\limits_{\tau_{1},\tau_{2}}
\sum\limits_{\sigma}I_{\tau_{2},\tau_{1}}(\sigma)=
\frac{|{\cal T}_{n}|^{2} }{n!}\sim \frac{e^{2\sqrt{n}}} {\sqrt{8\pi e n}}.
\end{equation}
Our results show that the average in (\ref{average}) is misleadingly large; if $n$
 is large, then for most 
permutations  $\sigma\in S_{n},$ one has 
\[ e^{(\frac{1}{2}-\epsilon)\log^{2}n}<\N_{n}(\sigma)< 
 e^{(\frac{1}{2}+\epsilon)\log^{2}n}. \] 

Another  consequence of the sum of indicators representation  (\ref{sumindicators})
is that $\max_{\sigma} \N_{n}(\sigma)=|{\cal T}_{n}|$. 
The  unique permutation  that attains the maximum is the
identity permutation that fixes all $n$ points.  At the other extreme, for $n\geq 2$,
$\min_{\sigma} \N_{n}(\sigma)=n-1$. The minimum  is attained only by the 
 $\frac{n!}{n-1}$ permutations that have
 a cycle of length $n-1$. These two extremal results are stated on page
161 of Lugo's thesis\cite{LugoT} and  are also  proved later in  \cite{this}. 
Lugo also conjectured, but did not prove, that $\N_{n}$ is asymptotically lognormal.

There is an extensive literature on formulas for the
 number of  ways to write a permutation as the product of two or more 
permutations with various restrictions on the conjugacy classes
of the factors of the product.  Without trying to review that literature,
we refer readers to \cite{GS}, \cite{Irving} as possible starting points.
For asymptotic problems, even an explicit formula can be quite useless
if it is too complicated.  However, as the authors in  \cite{GS} and  \cite{Irving}
point out,  formulas with non-negative terms tend to be more tractable. 
In this paper,  we make use of one such formula:

\begin{equation}
\label{Nproduct}
\N_{n}(\sigma)=\prod\limits_{k=1}^{n}\sum\limits_{j=0}^{\lfloor c_{k}/2\rfloor}
\frac{k^{c_{k}-j}c_{k}!}{2^{j}j! (c_{k}-2j)! },
\end{equation}
 where $c_{k}=c_{k}(\sigma)$ denotes the number of cycles of length $k$ 
that $\sigma$ has. As far as we know, the first 
complete proofs of (\ref{Nproduct})  are in  Petersen and Tenner \cite{PT} and
Lugo\cite{LugoT}. 

We use the formula (\ref{Nproduct}) to prove that, for most permutations $\sigma$, 
$\N_{n}(\sigma)$ can be well approximated  by $\B_{n}(\sigma)= \prod_{k}k^{c_{k}},$  
the product of the cycle lengths of $\sigma$.
The random variable   $\B_{n}$ has been studied by many authors,
 beginning with the work of Erd\"os and Tur\'an \cite{ET1}, \cite{ET4}. 
Asymptotic lognormality of 
$\N_{n}$ will be deduced from the known fact that $\B_{n}$
is asymptotically lognormal.

\section{Factorizations}
This section is more or less expository: we discuss the known factorization
 (\ref{Nproduct}).   
For each integer $x,$ let $\overline{x}=x-n\lfloor \frac{x}{n}\rfloor$
denote the integer remainder when $x$ is divided by $n$.
(The positive integer $n$ will be clear from context.)
Yang, Ellis, Mamakani,  and Ruskey \cite{YEMR} 
 proved the following lemma.  
\begin{lemma}
\label{onecycle} 
There are exactly $n$ ways to factor the  $n$-cycle $\sigma=(0,1,\dots ,n-1)$ as the 
product  of two involutions of $\lbrace 0,1,2,\dots ,n-1\rbrace$.  The $n$ factorizations 
are $\sigma = I_k \circ I_{k - 1}$,
  $1\leq k\leq n$, where    $I_{k}(x)= \overline{k-x}$   is the 
integer remainder when $k-x$ is divided by $n$.
\end{lemma}
Our notational preference for modular arithmetic 
is influenced by page 158 of \cite{GHR},
where the setting is different but the   factorization is similar.
In\cite{YEMR}, the proof  of lemma \ref{onecycle} is quite short, 
elementary, and easy to read.  
As we show in proposition  \ref{twocycles}  below, the proof of lemma \ref{onecycle} can
 be adapted to the product of two  $m$ cycles, and 
therefore can be used as the basis for an alternative  proof of  (\ref{Nproduct}).
Corresponding  lemmas appear in  \cite{LugoT} and \cite{PT}, 
 but the derivations there  are based on
 a graph theoretical insight and appear to be different from the
 proof that is presented here. 
\par
For any permutation 
$\sigma$, we can  apply  lemma \ref{onecycle} separately to each 
of the cycles of $\sigma$.
Therefore a consequence of lemma \ref{onecycle} is that   the 
product of the cycle lengths is
 a lower bound:  
\begin{equation}
\label{lowerbound}
\N_{n}(\sigma)\geq \B_{n}(\sigma).
\end{equation}
This inequality is  not sharp because, in the factorization 
$\sigma= \tau_{2}\circ \tau_{1}$,
 there is no requirement that the cycles of $\sigma $ are invariant under the 
 involutions $\tau_{1}$ and $\tau_{2}$.
For example, we can write $\sigma=(1,2,3)(4,5,6)$   as
 $\tau_{2}\circ \tau_{1}$,
where $\tau_{2}=(1,4)(2,6)(3,5)$ and $\tau_{1}=(1,6)(2,5)(3,4)$.
Both involutions  ``exchange'' the elements of $\lbrace 1,2,3\rbrace $
with those of $\lbrace 4,5,6\rbrace. $
 The next lemma asserts  that there are no other possibilities.

\begin{lemma}  
\label{swaplemma}  Suppose   $\mathcal{O}$ is  the set of points on a cycle of  $\sigma$, 
and  that $\sigma=\tau_{2}\circ\tau_{1}$ is a factorization of $\sigma$ into two involutions.
Then  $\tau_1(\mathcal{O}) = \tau_2(\mathcal{O})$, and
$\tau_{1}(\mathcal{O})$ is the set of points  on  a cycle of $\sigma$
  of length $|{\cal O}|$. 
\end{lemma}
\begin{proof}
Because each $\tau_{i}$
is a bijection, it is clear that  $|\tau_1(\mathcal{O})|=|\tau_2(\mathcal{O})|=
|\mathcal{O}|$. 

Suppose $y_{1},y _{2}$ are points in $\tau_{1}({\cal O}).$ We need to verify that 
$y_{1}$ and $y _{2}$ are on the same cycle of $\sigma$. 
Let $x_{1},x_{2}$ be the
their preimages on ${\cal O}: \tau_{1}(x_{i})=y_{i}$, $i=1,2$.
 Because $x_{1}$ and $x_{2}$ are on the same cycle ${\cal O}$, we 
have $x_{2}= \sigma^{\ell}     (x_{1})$ for some $\ell$.
But then $y_{2}=\tau_{1}( \sigma^{\ell}     (x_{1}))=
     \tau_{1} \circ (\tau_{2}\circ \tau_{1})^{\ell}(x_{1})=
(\tau_{1}\circ \tau_{2})^{\ell}\circ \tau_{1}(x_{1})
=\sigma^{-\ell}(y_{1}).$  Thus  $y_{1}$ and $y _{2}$ are on the same cycle,
and  $\tau_{1}({\cal O})$ is a single cycle
of length $|{\cal O}|. $

Finally, note that $\tau_{2}= \sigma \circ \tau_{1}. $
If $x\in \mathcal{O}$, then the set of points on the cycle of $\sigma$ that contains  $\tau_{2}(x)$ is  
$ \lbrace v: v=\sigma^{t}\circ \tau_{2}(x) \text{ for some } t\in {\mathbb Z}\rbrace =  
\lbrace v: v=\sigma^{t+1}\circ \tau_{1}(x) \text{ for some } t\in {\mathbb Z}\rbrace,$
and the latter set is the set of points on the cycle of $\sigma$ that contains $\tau_{1}(x).$
This proves that $\tau_1(\mathcal{O}) = \tau_2(\mathcal{O})$; the two involutions
both map ${\cal O}$ to the same cycle.
\end{proof}

\begin{definition}
Let $\mathcal{O}_1$ and $\mathcal{O}_2$ be two distinct sets of points on cycles of $\sigma.$ Two involutions $\tau_1, \tau_2$ \textit{exchange} $\mathcal{O}_1$ and $\mathcal{O}_2$ provided that $\sigma = \tau_2 \circ \tau_1$ and $\tau_1(\mathcal{O}_1) = \tau_2(\mathcal{O}_1) = \mathcal{O}_2.$
\end{definition}

\begin{lemma}
\label{twocycles} If $\sigma=(0,1,2,\dots ,n-1)(n,n+1,n+2,\dots ,2n-1)$, then there
are precisely $n$ ways to write $\sigma$ as a product of  two involutions of
$\lbrace 0,1,\dots ,2n-1\rbrace$ that exchange the two cycles of $\sigma$.
\end{lemma}
\vskip.2cm
{\sc Example:} If $n=5$, then one of the five factorizations is 
$(0,1,2,3,4)(5,6,7,8,9)=J_{3}\circ J_{2}$,\vskip0cm where
  $J_{3}=(0,8)(1,7)(2,6)(3,5)(4,9)$ and $J_{2}=(0,7)(1,6)(2,5)(3,9)(4,8)$.
\vskip.2cm
\begin{proof} Let $X=\lbrace 0,1,\dots ,2n-1\rbrace.$ For integral $k$, define  $J_{k}$
to be the involution whose $n$ transpositions  are $(x, n+\overline{k-x})$,   $x=0,1,2,\dots ,n-1$.
Note that  $J_{k}(x)=J_{k\pm n}(x),$ so  we are free to 
calculate the index $k$ modulo  $n$. 
Also note that
if $y=n+\overline{k-x}$, then $J_{k}(y)=\overline{x}.$
  Hence it  is straightforward to 
 verify that, for any integer $k$, 
 $\sigma=J_{k}\circ J_{k-1}$. 
Since there are $n$ choices for $\overline{k}$, this proves that
there {\sl at least} $n$ of the factorizations.

Now suppose $\sigma=S\circ T$ for some involutions $S$ and $T$ on $X$, and suppose $S$ and $T$ 
exchange the two cycles of $\sigma$.  Because $S$ exchanges the cycles of $\sigma$, there must 
be some $k$ for which $S(0)=n+\overline{k}$.   
To prove the lemma, it suffices to prove that $S=J_{k}$ and $T=J_{k-1}$.
We  use induction to show that, for $0\leq i< n$, 
$S(i)=n+\overline{k-i}$ and $T(i)=n+\overline{k-1-i}.$

For the base case $i=0$, we  already have $S(0)=n+\overline{k}$.
Note that
 $T(n+\overline{k-1})=S^{2}\circ T(n+\overline{k-1}) =
S\circ \sigma(n+\overline{k-1})=S(n+\overline{k})=0$. Therefore  $T(0)=n+\overline{k-1}$. 
This completes  the base case $i=0$.

Now let $0<i<n-1$, and assume the inductive hypothesis.
Since $i=\sigma(i-1)=ST(i-1)$, we have  
\[ S(i)=S^{2}\circ T(i-1) =T(i-1)\underbrace{=}_{\text{ind.hypoth.}}n+
\overline{k-1-(i-1)}=n+\overline{k-i}.\]
Similarly
\[  T(n+\overline{k-i-1})=S^{2}\circ T(n+\overline{k-i-1}       )=
 S\circ \sigma  (n+\overline{k-i-1})   =     S(n+\overline{k-i})=i.\]
 Therefore
\[ T(i)=n+\overline{k-1-i}.\]
\end{proof}

\vskip.2cm

For non-negative integers $m$ and $k$ define
\begin{equation}
\label{defVm}
 V_{m}(k)=\sum_{j = 0}^{\lfloor m/2 \rfloor} 
\frac{k^{- j}m!}{2^jj!(m - 2j)!} = \frac{He_{m}\!\left(i\sqrt{k}\right)}{\left(i\sqrt{k}\right)^{m}},
\end{equation}
where $He_m$ is the 
\lq\lq probabilists' Hermite polynomial\rq\rq
$He_{m}(x) = m!\sum_{r = 0}^{\lfloor m/2 \rfloor} 
\frac{(\textendash1)^r}{r!(m - 2r)!}\frac{x^{m - 2r}}{2^r}.$
We thank Victor Moll for pointing out this connection with the
 Hermite polynomials.  A less general verion appears as 
equation 2 of Moser and Wyman\cite{MoserWyman}. 

\begin{theorem} (Lugo, Petersen,Tenner)
\label{LPT}
If $c_k(\sigma)$ denotes the number of $k$-cycles that $\sigma \in S_n$ has, 
then  \[  \N_{n}(\sigma)=\B_{n}(\sigma) \prod\limits_{k=1}^{n} V_{c_{k}}(k)\]
\end{theorem}
\begin{proof}
By   lemma
\ref{swaplemma},  any involution factorization of $\sigma$ exchanges some 
number of pairs of cycles of the same size, and leaves the rest fixed. 
 For each $j \leq \lfloor c_{k}/2 \rfloor,$ there are precisely
$\frac{c_{k}!}{2^jj!(c_{k} - 2j)!}$
 ways to match $j$ pairs of $k$-cycles for swapping,  leaving  the remaining $c_{k} - 2j$
 $k$-cycles to be fixed. 
Once the $j$ pairs have been specified,  lemmas 
\ref{onecycle} and  \ref{twocycles}  show that there are $k^j \cdot k^{c_{k} - 2j}$ ways 
to factor the  $k$-cycles.  Hence, the total number of factorizations of $\sigma$ is
$\prod\limits_{k=1}^{n}\sum_{j = 0}^{\lfloor c_{k}/2 \rfloor} \frac{k^{c_{k} - j}c_{k}!}
{2^jj!(c_{k} - 2j)!}=  \prod\limits_{k=1}^{n}k^{c_{k}} V_{c_{k}}(k).  $
\end{proof}

\section{Approximation by $\B_{n}$}

Let $\T_{n}(\sigma)$ be the order of $\sigma$ as an element of the symmetric group, i.e.
 the  least common multiple
of the cycle lengths.
The asymptotic distribution of $\T_{n}$
was deduced from that  of $\B_{n}$.
 (See equation 14.4 of \cite{ET1}, section 7 of \cite{Best}, and 
lemma 2 of \cite{ATdp}.)   A similar strategy is used in this paper. 
The goal of this section is to prove that $\B_{n}$ can serve as proxy for  $\N_{n}$.

 The following deterministic 
 lemma supplies a sufficient condition on $\sigma$ that,  when satisfied, imposes
 a bound on the error of the approximation.
\vfill\eject
\begin{lemma}  
\label{deterministic}
Suppose $\xi \geq 1$ and that, for every integer $k > \xi,$ 
we have $c_k(\sigma) \leq 1.$ Also assume that, 
for every positive integer $k$, $c_k(\sigma) \leq \xi.$ 
Then there is a constant $c > 0,$ not dependent on $\sigma$ nor $\xi,$ such that
 $\B_{n}(\sigma) \leq \N_{n}(\sigma) \leq   \B_{n}(\sigma) \cdot \left( c\xi^{\xi}\right)^{\xi}$
\end{lemma}

\begin{proof} 
We already have the lower bound  (See equation \ref{lowerbound}).
Observe that
$V_0(k)= 1$ and $V_1(k) = 1$ for all $k \in [n].$ For $2 \leq m < \xi$ and $1 \leq k \leq \xi$, 
a very crude bound for $V_m(k)$ suffices. For example, by Stirling's formula we see
 that for $2 \leq m < \xi,$

\[ V_m(k) \leq m!\sum\limits_{j=0}^{\lfloor m/2 \rfloor} \frac{1}{(2k)^jj!} \leq m! e^{\frac{1}{2k}} < cm^{m},\]
 where $c$ is a positive constant independent of $k$ and $m$. 
 By assumption   $c_k(\sigma) \leq \xi$  for all $k\leq \xi$. Therefore
\[ \N_{n}(\sigma)\leq  \B_{n}(\sigma) \cdot \left(\prod_{1 \leq k \leq \xi} V_{c_k(\sigma)}(k)\right)\leq 
 \B_{n}(\sigma) \cdot \left( c\xi^{\xi}\right)^{\xi}\]
\end{proof}

Clearly $ \B_{n}(\sigma)$ is not {\sl always} a good approximation for $\N_{n}(\sigma)$.
For example, if $\sigma$ is the identity permutation with $n$ cycles of length one, 
then $\log \B_{n}(\sigma)=0$ and  $\log \N_{n}(\sigma)\sim \frac{n}{2}\log n$.
There is a tradeoff when applying  lemma \ref{deterministic}. The parameter 
$\xi=\xi(n)$ must be sufficently large so that  most permutations satisfy the hypotheses.
However  the larger $\xi$ is,  the cruder the bound.
The next two lemmas make this precise.

\begin{lemma}
\label{largeks}
If $\xi = \xi(n) \rightarrow \infty$ as $n \to \infty,$ and if $\mathbb{P}_{n} $ is the uniform 
probability measure on $S_{n}$, then
$\mathbb{P}_{n}\!\left(c_{k}\geq 2\ \text{for some}\ k \geq \xi\right) = O(\frac{1}{\xi}) $.
\end{lemma}

\begin{proof}
For any choice of $\xi,$ Boole's inequality  implies that
\begin{equation}
\label{boole}
\mathbb{P}_n(c_k \geq 2\ \text{for some}\ k \geq \xi) \leq \sum_{k \geq \xi} 
\mathbb{P}_n(c_k \geq 2) = \sum_{k =\lceil \xi\rceil }^{\lfloor\frac{n}{2}\rfloor }
\left[1 - \mathbb{P}_n(c_k = 0) - \mathbb{P}_n(c_k = 1)\right].
\end{equation}\
It is well known that the probabilities   $\mathbb{P}_n(c_k = j)$ can be calculated using the
 Principle of Inclusion Exclusion, and that  the alternating inequalites yield upper  and 
lower bounds. (See also chapter 5 of Sachkov\cite{Sachkov} for the 
\lq\lq generatingfunctionological\rq\rq approach).
Thus
\begin{equation}
\label{jzero}
\mathbb{P}_n(c_k = 0) = \sum_{j = 0}^{\lfloor \frac{n}{k} \rfloor} 
(\textendash1)^j\frac{1}{j!k^j} \geq 1 - \frac{1}{k}
\end{equation}\
\noindent and

\begin{equation}
\label{jone}
\mathbb{P}_n(c_k = 1) = \frac{1}{k}\sum_{j = 0}^{\lfloor n/k - 1 \rfloor}
 (\textendash1)^j\frac{1}{j!k^j} \geq \frac{1}{k}\left(1 - \frac{1}{k}\right).
\end{equation}\
Putting (\ref{jzero}) and (\ref{jone}) into (\ref{boole}), we get
\[ \mathbb{P}_n(c_k \geq 2\ \text{for some}\ k \geq \xi)
 \leq \sum_{k=\lceil \xi\rceil } ^{\lfloor \frac{n}{2}\rfloor} \left[1 - \left(1 - \frac{1}{k}\right) -
 \frac{1}{k}\left(1 - \frac{1}{k}\right)\right] = \sum_{k =
 \lceil \xi \rceil}^{\lfloor \frac{n}{2}\rfloor} \frac{1}{k^2} = O\!\left(\frac{1}{\xi}\right).\]
\end{proof}

The second hypothesis is even more likely to hold.
\begin{lemma}
\label{tvlemma}
If $\xi=\xi(n)\rightarrow\infty$, then
 ${\mathbb P}_{n}(c_{k}\geq \xi  \text{ for\ some\  } k \leq \xi    )=
O((\frac{\xi}{e^\xi}+\frac{\xi}{n})).$
\end{lemma}

\begin{proof}
Let $ {\mathbf Z}_{k},\ \xi\leq k\leq n$ be a sequence of independent  
Poisson($1/k$) random variables.
By theorem 4 of \cite{AT},
$\PP_{n}( c_k \leq \xi \text{ for all\ } k \leq \xi )=\Pr( {\mathbf Z}_k 
 \leq \xi   \text{ for all\ } k \leq \xi )+O(\frac{\xi}{n}).$\
Standard estimates using Markov's inequality and moment generating functions shows that 
this probability is small:
\[ \Pr({\mathbf Z}_k  \geq \xi )= \Pr(e^{{\mathbf Z}_k } \geq e^{\xi} )\]
\[ \leq  \frac{  \E(e^{Z_{k}})}{e^{\xi}}=\frac{ e^{\frac{1}{k}(e-1)}}
{e^{\xi}}< \frac{8}{e^{\xi}}.  \]
Therefore 
\[  \Pr({\mathbf Z}_k \leq  \xi  \text{ for all }\ k \leq \xi )\geq
 \left(1- \frac{8}{e^{\xi}}\right)^{\xi}=1-O\!\left(\frac{\xi}{e^{\xi}}\right).\]
\end{proof}

\section{The Asymptotic Lognormality of $\N$}

It is well known that $\B_{n}$ is asymptotically lognormal.

\begin{lemma} 
\label{ET}
(Erd\"{o}s and Tur\'{a}n) 
For any real number $x,$
\[\lim_{n \to \infty} \mathbb{P}_{n}(\log\B_{n}(\sigma) \leq \mu_{n} 
+ x\sigma_n)= \Phi(x)\]
where
$\mu_n =\sum\limits_{k=1}^{n} \frac{\log k}{k} \sim  \dfrac{1}{2}\log^2\!n$,
$\sigma^{2}_n =  \sum\limits_{k=1}^{n} \frac{\log^{2}k}{k}\sim $
$\dfrac{1}{3}\log^{3}\!n$,  and $ \Phi(x) = \frac{1}{\sqrt{2\pi}} $
$\int_{\textendash\infty}^{x}e^{\textendash t^2/2}\,dt $
\end{lemma}

\begin{remark}
The first proof lemma \ref{ET} is in the work of Erd\"{o}s and Tur\'{a}n 
\cite{ET4}. Alternative proofs, as well as stronger and more general results have 
been proved using quite varied techniques. 
See, for example, \cite{ABT}, \cite{ABTlim}, \cite{ATdp},\cite{DPittel},\cite{Man}.
\end{remark}

\begin{theorem}
\label{mainthm}
$\mathbb{P}_n(\log\N_{n}(\sigma) \leq \mu_n + x\sigma_n) = \Phi(x) + o(1).$
\end{theorem}

\begin{proof}
Because $\N_{n}(\sigma) \geq \B_{n}(\sigma)$ for all $\sigma \in S_n,$ 
one direction is an immediate consequence of lemma \ref{ET}.

\begin{equation}
\mathbb{P}_n(\log\N_{n} \leq \mu_n + x\sigma_n) \leq 
\mathbb{P}_n(\log\B_{n} \leq \mu_n + x\sigma_n) = \Phi(x) + o(1).
\end{equation}\
 
\noindent For the other direction, we use the continuity of $\Phi$ and the bound
 $\N_{n}(\sigma) \leq (c\xi^{\xi})^{\xi} \B_{n}(\sigma)$ from
 Lemma \ref{deterministic}, which, 
due to lemma \ref{largeks} and lemma \ref{tvlemma}, holds with probability 
$1 - O(\frac{1}{\xi}+\frac{\xi}{n}+\frac{\xi}{e^{\xi}}).$

In more detail, let $\epsilon >0$ be a fixed but arbitrarily small postive number. 
We can choose $\delta>0$ so that $|\Phi(x) - \Phi(a)| < \epsilon$ 
whenever $|x - a| < \delta.$
If we choose $\xi=\sqrt{\log n},$  then we have
$ \log\!\left( (c\xi^{\xi} )^{\xi}\right)=o(\sigma_{n}).$
Therefore we can choose  $N_{\epsilon}$ so that, for all $n\geq N_{\epsilon}$, 
$  \log\!\left( (c\xi^{\xi} )^{\xi}\right)< \frac{\delta \sigma_{n}}{2}.$
But then

\begin{align}
\mathbb{P}_n(\log\N(\sigma) \leq \mu_n + x\sigma_n) &\geq 
\mathbb{P}_n\!\left(\log\B(\sigma) +  \log\!\left( (c\xi^{\xi} )^{\xi}\right)
\leq \mu_n + x\sigma_n\right)
\\ &\geq \mathbb{P}_n\!\left(\log\B(\sigma) + \frac{\delta\sigma_n}{2} 
\leq \mu_n + x\sigma_n\right)
\\ &= \mathbb{P}_n\!\left(\log\B(\sigma) \leq \mu_n +
 \left(x - \frac{\delta}{2}\right)\sigma_n\right)
\\ &= \Phi\!\left(x - \frac{\delta}{2}\right) + o(1) > \Phi(x) - \epsilon + o(1)
\end{align}\

\noindent Yet $\epsilon > 0$ was arbitrary, and so $\mathbb{P}_n(\log\N(\sigma) 
\leq \mu_n + x\sigma_n) \geq \Phi(x) + o(1).$
\end{proof}


\bibliographystyle{plain}

\begin{thebibliography}{99}
\bibitem{AM}  Amdeberhan, Tewodros and  Moll, Victor H.
Involutions and their progenies, arXiv:1406.2356 [math.CO].



\bibitem{ABT}
 {Arratia, Richard and Barbour, A. D. and Tavar{\'e}, Simon},
 {Logarithmic combinatorial structures: a probabilistic
              approach},
 {\sl EMS Monographs in Mathematics},
 (2003) ISBN = {3-03719-000-0}.  

\bibitem{ABTlim}
 {Arratia, R. and Barbour, A. D. and Tavar{\'e}, S.},
   {Limits of logarithmic combinatorial structures},
 {\sl Ann. Probab.}, {\bf 28}, (2000), {no.4}, {1620--1644}.
     
\bibitem{ATdp} {Arratia, Richard and Tavar{\'e}, Simon},
 {Limit theorems for combinatorial structures via discrete
              process approximations},
 {\sl Random Structures Algorithms},
 {\bf 3}, (1992), {no. 3}, {321--345}.

\bibitem{AT} 
 {Arratia, Richard and Tavar{\'e}, Simon},
 {The cycle structure of random permutations},
 {\sl Ann. Probab.},  {\bf 20},(1992),  {no.3}, {1567--1591}.

\bibitem{Best}
 {Best, M. R.},
  {The distribution of some variables on symmetric groups},
 {\sl Nederl. Akad. Wetensch. Proc. Ser. A {\bf 73}=Indag. Math.},  {\bf 32},
  (1970),
 {385--402}.
  
\bibitem{this}
 Burnette, Charles, Drexel University Doctoral Dissertation, in. prep.
\bibitem{CHM}
 {Chowla, S. and Herstein, I. N. and Moore, W. K.},
 {On recursions connected with symmetric groups. {I}},
 {\sl Canadian J. Math.},
 {\bf 3}   (1951),  {328-334}.

\bibitem{DPittel} 
 {DeLaurentis, J. M. and Pittel, B. G.},
 {Random permutations and {B}rownian motion},
 {\sl Pacific J. Math.},
  {\bf 119}, (1985), {no. 2}, {287--301}.

\bibitem{ET1}
  {Erd{\H{o}}s, P. and Tur{\'a}n, P.},
   {On some problems of a statistical group-theory. {I}},
 {\sl Z. Wahrscheinlichkeitstheorie und Verw. Gebiete},
    {\bf 4}, (1965), 175-186.


\bibitem{ET4} 
 {Erd{\H{o}}s, P. and Tur{\'a}n, P.},
 {On some problems of a statistical group-theory. {III}},
{\sl Acta Math. Acad. Sci. Hungar.}, {\bf 18}, (1967),  {309-320}.
     
\bibitem{GHR} {Gustafson, W. H. and Halmos, P. R. and Radjavi, H.},
  {Products of Involutions},
 {\sl Linear Algebra and Appl.},  {\bf 13}, (1976),
   no. {1/2},  {157-162}.

\bibitem{GS}
{Goupil, Alain and Schaeffer, Gilles},
 {Factoring {$n$}-cycles and counting maps of given genus},
 {\sl European J. Combin.}, {\bf 19},  {1998}, {no. 7}, {819--834}.
   

\bibitem{Irving}
{Irving, John},
 {On the number of factorizations of a full cycle},
 {\sl J. Combin. Theory Ser. A}, {\bf 113},
 {2006}, {no. 7}, {1549--1554}.
     
\bibitem{lugopaper}
 {Lugo, Michael},{Profiles of permutations}, {Electron. J. Combin.}, {\bf 16},  {no. 1},
 {Research Paper 99 } (2009)
     
\bibitem{LugoCompose} Lugo, Michael T.,The cycle structure of compositions of random 
involutions,  (2009),
arXiv:0911.3604 [math.CO].

\bibitem{LugoT}  {Lugo, Michael T.},
 {Profiles of large combinatorial structures},
{Thesis (Ph.D.)--University of Pennsylvania},
 {ProQuest LLC, Ann Arbor, MI},
  {2010},  ISBN = {978-1124-31808-0}.

\bibitem{Man}
 {Manstavi{\v{c}}ius, E.},
 {The {B}erry-{E}sseen bound in the theory of random
              permutations},
     {Ramanujan J.},
  {\bf 2},
 (1998),
   {no. 1-2},
 {185--199}.

     
      
\bibitem{MoserWyman} 
 {Moser, Leo and Wyman, Max},
 {On solutions of {$x^d=1$} in symmetric groups},
    {Canad. J. Math.}, {\bf 7},
 (1955) {159--168}.
     
\bibitem{MoserWyman-ae}
{Moser, Leo and Wyman, Max},
 {Asymptotic expansions}, {\sl Canad. J. Math.},
  {\bf 8},  {1956},  {225--233}.
      
\bibitem{PT}
 {Petersen, T. Kyle and Tenner, Bridget Eileen},
 {How to write a permutation as a product of involutions (and
  why you might care)},{\sl Integers},{\bf 13}, (2013), {Paper No. A63, 20}.
 
     
\bibitem{RV}
 {Roberts, John A. G. and Vivaldi, Franco},
 {A combinatorial model for reversible rational maps over finite
              fields}, {\sl Nonlinearity}, {\bf 22}, {2009}, {no. 8}, {1965--1982}.
\bibitem{SachkovVN}
 {Sachkov, V. N.},
 {Asymptotic formulas and limit distributions for combinatorial
              configurations generated by polynomials},
 {Discrete Math. Appl..}, {\bf 19}, {2007}, {no.4}, {319-330}.
   
\bibitem{Sachkov}
  {Sachkov, Vladimir N.},
 {Probabilistic methods in combinatorial analysis},
 {\sl Encyclopedia of Mathematics and its Applications},
 {\bf 56},
 {Cambridge University Press, Cambridge},
    {1997}, ISBN  {0-521-45512-X}.

 \bibitem{sequence} OEIS Sequence A000085,
Number of self-inverse permutations on n letters, also known as involutions; number of Young tableaux with n cells,
(Formerly M1221 N0469),
{\sl The On-Line Encyclopedia of Integer Sequences}, published electronically
 at http://oeis.org.



\bibitem{SV}
Schaeffer, Gilles and Vassilieva, Ekaterina,
{A bijective proof of {J}ackson's formula for the number of
              factorizations of a cycle},
 {\sl J. Combin. Theory Ser. A},
 {\bf 115},  (2008)  {no.6},{903--924}.

\bibitem{Wilf}
 {Wilf, Herbert S.},
  {The asymptotics of {$e^{P(z)}$} and the number of elements 
  of each order in {$S_n$}}, {\sl Bull. Amer. Math. Soc. (N.S.)},
 {\bf 15},(1986), {no. 2}, {228--232}.
    

\bibitem{YEMR}
 {Yang, Qingxuan and Ellis, John and Mamakani, Khalegh and
              Ruskey, Frank},  {In-place permuting and perfect shuffling using involutions},
 {\sl Inform. Process. Lett.}, {\bf 113} (2013), no. {10-11}, {386--391}.
    
\end{thebibliography}

\end{document}